\documentclass[reqno,11pt]{amsart}
\usepackage[a4paper,left=30mm,right=30mm,top=30mm,bottom=30mm,marginpar=20mm]{geometry}
\usepackage{bbm}
\usepackage{enumerate}
\usepackage{amsmath,mathtools}
\usepackage{amsmath}
\usepackage{float,enumerate}
\mathtoolsset{showonlyrefs}
\def\eps{\varepsilon}

\def\R{{\mathbb R}}

\newcommand{\diverg}{\operatorname{div}}

\newcommand{\norm}[1]{\left\Vert#1\right\Vert}

\newcommand{\set}[1]{\left\{#1\right\}}

\newcommand{\C}{\mathbb C}

\newcommand{\dd}{\;\mathrm{d}}

\DeclareMathOperator{\supp}{supp}

\newtheorem{lemma}{Lemma}
\newtheorem{theorem}[lemma]{Theorem}
\newtheorem{corollary}[lemma]{Corollary}
\newtheorem{proposition}[lemma]{Proposition}
\theoremstyle{definition}
\newtheorem{definition}[lemma]{Definition}
\theoremstyle{remark}
\newtheorem{remark}[lemma]{Remark}
\usepackage{hyperref}

\usepackage{mathtools}

\makeatletter
\DeclareRobustCommand\widecheck[1]{{\mathpalette\@widecheck{#1}}}
\def\@widecheck#1#2{%
    \setbox\z@\hbox{\m@th$#1#2$}%
    \setbox\tw@\hbox{\m@th$#1%
       \widehat{%
          \vrule\@width\z@\@height\ht\z@
          \vrule\@height\z@\@width\wd\z@}$}%
    \dp\tw@-\ht\z@
    \@tempdima\ht\z@ \advance\@tempdima2\ht\tw@ \divide\@tempdima\thr@@
    \setbox\tw@\hbox{%
       \raise\@tempdima\hbox{\scalebox{1}[-1]{\lower\@tempdima\box
\tw@}}}%
    {\ooalign{\box\tw@ \cr \box\z@}}}
\makeatother

\begin{document}

\title[Non-Unique Admissible Solutions of the Compressible Euler Equations]%
{Non-Unique Admissible Weak Solutions of the Compressible Euler Equations with Compact Support in Space}
\author[I. Akramov]{Ibrokhimbek Akramov}
\address{\textit{Ibrokhimbek Akramov:} Institute of Applied Analysis, Ulm University, Helmholtzstraße~18, 89081 Ulm, Germany}
\email{ibrokhimbek.akramov@uni-ulm.de}

\author[E. Wiedemann]{Emil Wiedemann}
\address{\textit{Emil Wiedemann:} Institute of Applied Analysis, Ulm University, Helmholtzstraße~18, 89081 Ulm, Germany}
\email{emil.wiedemann@uni-ulm.de}
\begin{abstract}
This paper is concerned with the existence of compactly supported admissible solutions to the Cauchy problem for the isentropic compressible Euler equations. In more than one space dimension, convex integration techniques developed
by De Lellis-Sz\'ekelyhidi and Chiodaroli enable us to prove failure of uniqueness on a finite time-interval for admissible solutions starting from any continuously differentiable initial density and suitably constructed bounded initial momenta. In particular, this extends Chiodaroli's work from periodic boundary conditions to bounded domains or the whole space.
\end{abstract}
\maketitle
\section{Introduction}
In this paper, we consider the isentropic compressible Euler system consisting of $(n+1)$ equations
\begin{equation}\label{comm.euler}
\begin{cases}
\partial_t\rho+\diverg_x m=0, \\
\partial_t m+\diverg_x\left(\frac{m\otimes m}{\rho}\right)+\nabla_x[p(\rho)]=0,\\
\rho(\cdot, 0)=\rho^0,\quad m(\cdot, 0)=m^0,
\end{cases}
\end{equation}
where $m$ is the momentum and $\rho$ is the density of a gas. The pressure $p$ is a function of $\rho$, which is determined from the constitutive thermodynamic relations of the gas and is assumed continuously differentiable on $(0, \infty)$ with  $p'(\rho)> 0$  throughout the paper. The latter condition makes the system strictly hyperbolic on the set of admissible values $\{\rho>0\}$ (cf.~\cite{Dafermos}). Furthermore, thermodynamically admissible processes must satisfy an additional constraint given by the energy inequality
\begin{equation}\label{eq1.2}
\partial_t\left(\rho\eps(\rho)+\frac12\frac{|m|^2}{\rho}\right)+\diverg_x\left[\left(\eps(\rho)+\frac12\frac{|m|^2}{\rho^2}+\frac{p(\rho)}{\rho}\right)m\right]\leq 0,
\end{equation}
where the internal energy $\eps:\R^+\to\R$ is given through the law $p(r)=r^2\eps'(r)$. 

Let $T$ be a fixed positive time. By a bounded weak solution of \eqref{comm.euler} we mean a pair $(\rho, m)\in L^\infty(\R^n\times(0,T))$ 
such that the following identities hold for every test function $\psi\in C_c^\infty([0, T); C_c^\infty(\R^n))$, and for any vector field $\phi\in C_c^\infty([0, T); C_c^\infty(\R^n))$:
\begin{equation}\label{eq2'.1}
\int_0^T\int_{\R^n}\left[\rho\partial_t\psi+m\cdot\nabla_x\psi\right]\dd x\dd t+\int_{\R^n}\rho^0(x)\psi(x, 0)\dd x=0
\end{equation}
\begin{equation}\label{eq2.4ch}
\begin{split}
\int_0^T\int_{\R^n}&\left[m\partial_t\phi+\left\langle\frac{m\otimes m}{\rho}, \nabla_x\phi \right\rangle+p(\rho)\diverg_x\phi \right]\dd x\dd t\\&\quad+\int_{\R^n}m^0\phi(x,0)\dd x=0.
\end{split}
\end{equation}
It is tacitly assumed, as part of the definition, that all the integrals are well-defined (if $\rho$ is bounded below by a positive constant, this will automatically be the case). 

Weak solutions satisfying~\eqref{eq1.2} in the sense of distributions represent a special case of \emph{entropy solutions}, as have been studied for decades in the theory of hyperbolic conservation laws. They were long viewed as the solution paradigm for conservation laws, giving rise to a very satisfactory well-posedness theory at least in the scalar case~\cite{kruzhkov}. However, the expected uniqueness of entropy solutions for the Cauchy problem for systems of conservation laws was disproved in the groundbreaking work of De Lellis-Sz\'ekelyhidi~\cite{Delelissekhidieuler}, who gave examples of non-unique entropy solutions for~\eqref{comm.euler} with piecewise constant density. Their convex integration scheme was later refined, e.g., in~\cite{Chiodaroli, CDK, feireisl16, MarkKling} to yield larger classes of non-unique solutions, particularly for the Riemann problem.

Our aim in this paper is to construct non-unique compactly supported solutions to~\eqref{comm.euler}, in the sense that the momentum $m$ has compact support in space for every time, and the density $\rho$ is constant outside a compact set in space for every time. More precisely, we construct such solutions which are \emph{semi-stationary}, i.e., $\rho$ is independent of time. The semi-stationary system has been studied in~\cite{Chiodaroli}, but only under periodic boundary conditions. Although the specific examples of~\cite{Delelissekhidieuler} already have compact support in the mentioned sense, we extend the range of compactly supported non-unique solutions to a much wider class of (smooth) initial densities. Note that this allows us to view our solutions either as solutions on all of $\R^n$ or as solutions on any bounded domain containing the support of $m$ and $\rho-\overline{\rho}$. Our work can thus be seen as a proper extension of the results in~\cite{Chiodaroli}.

More precisely, our results are as follows:




\begin{theorem}\label{Th2}
Let $n\geq 2$, $\Omega\subset\R^n$ a bounded open set, $T>0$, and $\Omega'\supset\supset\Omega$ locally Lipschitz. Assume that $\rho^0\in C^1(\R^n)$ is a positive function satisfying $\rho^0(x)=\overline{\rho}>0$ for $x\in \R^n\backslash\Omega$ and the pressure $p\in C^1(\R^n)$ with $\int_{\Omega}p(\rho_0(x))\dd x=p(\overline{\rho})|\Omega|$.
There exists a bounded initial momentum $m^0$ with $\supp(m^0)\subset \Omega'$ for which there are infinitely many weak solutions $(\rho, m)\in C^1(\R^n)\times C([0, T); H_w(\R^n))$ of 
\begin{equation}\label{com.euler}
\begin{cases}
\diverg_x m=0, \\
\partial_t m+\diverg_x\left(\frac{m\otimes m}{\rho}\right)+\nabla_x[p(\rho)]=0,\\
 m(\cdot, 0)=m^0
\end{cases}
\end{equation}
on $\R^n\times[0, T)$ with density $\rho(x)=\rho^0(x)$. Moreover, the obtained weak solutions $m$ satisfy
\begin{equation*}
|m(x,t)|^2=\rho^0(x)\chi(t)\mathbbm{1}_{\Omega'}\quad a.e. \quad\text{in}\quad\R^n\times[0, T),
\end{equation*}
\begin{equation}\label{eq2.10}
|m^0(x)|^2=\rho^0(x)\chi(0)\mathbbm{1}_{\Omega'}\quad a.e. \quad\text{in}\quad\R^n
\end{equation}
for some smooth function $\chi:\R\to\R$.
\end{theorem}
\begin{theorem}\label{Th3}
Under the same assumptions of Theorem \ref{Th2}, there exists a maximal time $\overline{T}>0$ such that the weak solutions $(\rho, m)$ of \eqref{com.euler} (coming from Theorem \ref{Th2}) satisfy the admissibility condition:
\begin{equation}\label{eq2.5ch}
\begin{split}
&\int_0^T\int_{\R^n}\left[\left(\rho\eps(\rho)+\frac12\frac{|m|^2}{\rho}\right)\partial_t\varphi+\left(\eps(\rho)+\frac12\frac{|m|^2}{\rho^2}+\frac{p(\rho)}{\rho}\right)m\cdot\nabla_x\varphi \right]\dd x\dd t\\&+\int_{\R^n}\left(\rho^0\eps(\rho^0)+\frac12\frac{|m^0|^2}{\rho}\right)\varphi(\cdot, 0)\dd x\geq 0
\end{split}
\end{equation}
for every nonnegative $\varphi\in C_c^\infty([0, T);C_c^\infty(\R^n))$ with $T\leq\overline{T}$.

\end{theorem}
\begin{corollary}\label{Th1}
Let $n\ge 2$ and $\Omega\subset\mathbb{R}^n$ be a nonempty bounded open set. Assume that $\rho^0\in C^1(\R^n)$ satisfies
\begin{itemize}
\item  $\rho^0(x)>0$ for any $x\in\mathbb{R}^n$,
\item $\rho^0(x)=\overline{\rho}$ for $x\in\mathbb{R}^n\backslash \Omega$.
\end{itemize}
Let $p\in C^1$ be given function such that $\int_\Omega p(\rho(x))\dd x=p(\overline{\rho})|\Omega|$. Then there exist $\Omega'\supset\Omega$, $m^0$ and a positive time $\overline{T}$ such that $\supp m^0\subset\Omega'$, $\diverg_xm^0=0$ for which there exist infinitely many $m$ such that $\supp m(\cdot, t)\subset \Omega'$ for $t\in [0, \overline{T})$ and $(\rho, m)$ is an admissible solution of \eqref{comm.euler} on $\R^n\times[0, \overline{T})$ with $\rho(x, t) = \rho_0(x)\mathbbm{1}_{[0,\overline{T})}(t)\in C^1(\R^n\times[0, \overline{T}))$.
\end{corollary}

The results of this article are partially motivated by possible applications to the 3D axisymmetric incompressible Euler equations (work in progress). Let us mention that, in the incompressible context, spatially compactly supported solutions can readily be constructed in $L^\infty$, but for (H\"older-)continuous solutions this question has apparently only been dealt with in~\cite{IsettOh}. In the compressible setting, continuous non-unique solutions are not yet available. To our knowledge, the present paper yields the first examples of compactly supported non-unique solutions to~\eqref{comm.euler} whose density is not piecewise constant.

Methodologically, the general convex integration scheme employed here follows the ones in~\cite{Delelissekhidieuler,Chiodaroli}. However, in the absence of periodic boundary conditions, we can no longer use the subsolutions constructed in~\cite{Chiodaroli}. Instead, we use elliptic theory to obtain compactly supported solutions for a Poisson equation whose right-hand side has a specific structure (Lemma~\ref{LemPWS}), and from this we build our subsolutions. We can then show, albeit with more effort than in~\cite{Chiodaroli}, that these subsolutions give rise to weak solutions satsfying the energy inequality~\eqref{eq1.2}.      

The paper is organized as follows. Section \ref{Prelim} contains some notions and facts that we will need later on. In Section \ref{Geo}, we recall some known results from convex integration theory and adapt some of them to our settings. The purpose of Section \ref{criterion} is to provide a general criterion on the existence of weak solutions to \eqref{comm.euler} for a given initial data. Section \ref{main} proves the non-uniqueness results by mainly employing Proposition \ref{Prop4.1} and Proposition \ref{p7.3}.
\section{Preliminaries} \label{Prelim}
In this section we establish some facts and notions which will be useful later on. For any set $\Omega\subset\R^n$ and $\eps>0$, define $\Omega^\eps:=\{y\in\R^n: \operatorname{dist}(y,\Omega)<\eps\}$.

For a domain $\Omega\subset\R^n$, we denote by $L^2_w(\Omega)$ the space of measurable, square integrable functions equipped with the weak topology. By $H(\Omega)$, we denote the space of solenoidal $L^2$-vectorfields $\Omega\to\R^n$, and $H_w(\Omega)$ is the same space but with the weak topology.

Next, recall the Paley-Wiener-Schwartz Theorem (for more details, see \cite{ReedSimon,Hoermander90}):
\begin{theorem}[Paley-Wiener-Schwartz]\label{PWS}
If $u$ is a distribution of order $N$ with support contained in a closed ball $\overline{B_r(0)}\subset\R^n$, then its Fourier transform $\widehat{u}$ 
can be extended to an entire function in $\C^n$ satisfying
\begin{equation*}
|\widehat{u}(\xi)|\leq C(1+|\xi|)^Ne^{r\left|\mbox{Im }\xi\right|},\quad \xi\in\C^n.
\end{equation*}

Conversely, an entire function in $\C^n$ meeting such an estimate is the Fourier transform of a distribution of order $N$ supported inside $\overline{B_r(0)}\subset\R^n$.
\end{theorem}
Let $\omega:\R^n\to\R$ be a smooth spherically symmetric function satisfying the following conditions:
$w(x)$ is constant for $|x|\leq\frac12$, $\supp(\omega)\subset B_1(0)$, $\omega(x)\geq 0$ and 
\begin{equation*}
\int_{\R^n}\omega(x)\dd x=1.
\end{equation*}

Denote $\displaystyle\omega^\eps(x)=\frac{1}{\eps^n}\omega\left(\frac{x}{\eps}\right)$ for $\eps>0$.
Since $\rho^0(x)=\overline{\rho}$ for $x\notin\Omega$, we have $p(\rho^0)=p(\overline{\rho})=\overline{p}$. Let $p_1(x)=p(\rho^0(x))-\overline{p}$. Then $\supp p_1\subset\Omega$ and

\begin{equation*}
\int_{\R^n}p_1(x)\dd x=\int_{\Omega}p_1(x)\dd x=0.
\end{equation*}

The main point about the following lemma is that, for right hand sides of a specific form, Poisson's equation admits solutions with compact support.
\begin{lemma}\label{LemPWS} Let  $p^\eps$ be defined by $p^\eps(x):=p_1(x)-p_1*\omega^\varepsilon(x)$, so that $\supp(p^\eps)\subset{\Omega^\eps}$. Then, there exists $u\in C^{2,\alpha}_c(\R^n)$ for every $0<\alpha<1$ such that
\begin{equation*}
\Delta u=p^\eps\quad\text{and}\quad\supp u\subset\overline{\Omega^\eps}.
\end{equation*}
\end{lemma}
\begin{proof}
By the Paley-Wiener-Schwartz Theorem (Theorem~\ref{PWS}), the Fourier transform $\widehat{p^\eps}$ is an analytic function, and we have
\begin{equation*}
\widehat{p^\eps}=\widehat{p_1}-\widehat{p_1}\widehat{\omega^\eps}=\widehat{p_1}(1-\widehat{\omega^\eps}).
\end{equation*}
Note that $\widehat{\omega^\eps}(0)=1$. Since $\omega^\eps$ is a spherically symmetric function, so is $\widehat{\omega^\eps}(\xi)$. So there exists a single variable analytic function $g$ on $\C$   such that
\begin{equation*}
\widehat{\omega^\eps}(\xi)=g(|\xi|^2),\quad\text{where}\quad  |\xi|^2=\xi_1^2+\dots+\xi_n^2.
\end{equation*}
Again due to the Paley-Wiener-Schwartz Theorem, as $\supp\omega^\eps\subset B_\eps(0)$, we have the estimate
\begin{equation*}
|\widehat{\omega^\eps}(\xi)|\leq C\exp(\eps|\mbox{Im }\xi|)\quad\text{for}\quad\xi\in \C^n.
\end{equation*}
Moreover, since $g(0)=1$, there exists an analytic function $g_1$ such that $1-g(z)=-zg_1(z)$. Therefore,
\begin{equation*}
1-g(|\xi|^2)=-|\xi|^2g_1(|\xi|^2),
\end{equation*}
and so we have
\begin{equation*}
\left||\xi|^2g_1(|\xi|^2)\right|\leq C\exp(\eps|\mbox{Im }\xi|).
\end{equation*}
We claim that
\begin{equation*}
\left|g_1(|\xi|^2)\right|\leq C_2\exp(\eps|\mbox{Im }\xi|).
\end{equation*}
Indeed, $g_1(z)$ is bounded for $|z|\leq 1$, say $|g_1(z)|\leq A$ for $|z|\leq 1$.
If $||\xi|^2|\geq 1$, then obviously we have
\begin{equation*}
    \left|g_1(|\xi|^2)\right|\leq C\exp(\eps|\mbox{Im }\xi|).
\end{equation*}
Thus, combining the two obtained bounds we get
\begin{equation*}
\left|g_1(|\xi|^2)\right|\leq \max\{A, C\}\exp(\eps|\mbox{Im }\xi|)\quad\text{for all}\quad\xi\in \C^n.
\end{equation*}
Invoking once again the Paley-Wiener-Schwartz Theorem, the inverse Fourier transform $\widecheck{g_1}$ is concentrated in $\overline{B_\eps(0)}$, i.e., $\supp\widecheck{g_1}\subset\overline{B_\eps(0)} $. Hence, we obtain
\begin{equation*}
\widehat{p^\eps}=-|\xi|^2\widehat{p_1} g_1(|\xi|^2).
\end{equation*}
Let $u$ be a distribution defined by its Fourier transform
\begin{equation*}
\widehat{u}(\xi)=\widehat{p_1} g_1(|\xi|^2).
\end{equation*}
Then $u=p_1*\widecheck{g_1}$, and from $\supp \widecheck{g_1}\subset \overline{B_\eps(0)}$ and $\supp p_1\subset\Omega$ it follows that $\supp u\subset\overline{\Omega^\eps}$.
Moreover, we have
\begin{equation*}
\widehat{\Delta u}=-|\xi|^2\widehat{u}(\xi)=-|\xi|^2\widehat{p_1} g_1(|\xi|^2)=\widehat{p^\eps}
\end{equation*}
and thus $\Delta u=p^\eps$ as desired. Finally, the regularity $u\in C^{2,\alpha}$ follows from standard Schauder theory, as $p^\eps$ is continuously differentiable.

The lemma is proved.
\end{proof}
\begin{remark}
Note that, generally, $p\in C_c^\infty(\R^n)$ does not imply that the unique decaying solution $u$ of $\Delta u=p$ is compactly supported. For instance, if $p\geq 0$ is not identically zero, the the corresponding solution $u$ does not have compact support. Indeed, any compactly supported subharmonic function vanishes identically by the maximum principle.
\end{remark}

\begin{lemma}\label{exdiv} Let $n\geq 2$. Let $\Omega'$
be a bounded locally Lipschitz domain containing $\overline{\Omega^\eps}\subset\R^n$.
If $p\in C_c^\infty(\Omega^\eps)$ satisfies the compatibility condition
 $$\int_{\Omega'}p(x)\dd  x=0,$$
then there exists $(\phi_j)_{j=1}^n\subset C_c^\infty(\Omega')$ such that
\begin{equation*}
\sum_{j=1}^{n}\partial_j\phi_j=p.
\end{equation*}
\end{lemma}
This lemma is an immediate consequence of~\cite[Theorem III.3.3]{Galdi}.
\begin{remark}
Note that the regularity assumption on $\Omega'$ 
can be further weakened, for instance, to a bounded domain of $\R^n$, such that $\Omega'=\cup_{k=1}^N\Omega'_{k}$, $N\geq 1$,
where each $\Omega'_{k}$ is star-shaped with respect to some open ball $B_k$ with $\overline{B_k}\subset\Omega'_k$ and  
 $$\int_{\Omega'_{k}}p(x)\dd  x=0,\quad \mbox{for any}\quad k=1,\dots, N$$
(see Theorem III3.1 \cite{Galdi}).
\end{remark}

Let $\mathcal{S}_n$ be the space of symmetric $n\times n$ matrices and let $\mathcal{S}_0^n$ be the subspace of $\mathcal{S}_n$ with null trace. Further we denote by $I_n$ the $n\times n$ identity matrix.
\begin{proposition}\label{Prop.forU2}
Let $p\in C_c^\infty(\Omega^\eps)$ such that $\int_{\Omega'}p(x)\dd  x=0$, with $\Omega'$ being given as in Lemma \ref{exdiv}.
Then there exists a pair $(m, U)$ of a vector field and a matrix field with values in $\mathcal{S}_0^n$ satisfying the following conditions:
\begin{enumerate}[(i)]
\item $(m(t), U)\in C_c^\infty(\Omega')$ for each $t\in\R$,
\item $\diverg m=0$,
\item $\partial_t m+\diverg U+\nabla p=0$,
\item $m$ is linear in $t$, and $U$ does not depend on $t$.
\end{enumerate}
\end{proposition}
\begin{proof}
Define the matrix field $A$ by
\begin{equation*}
A:=\left[\frac{n}{1-n}\left(\partial_i\phi_j-\frac{p}{n}\delta_{ij}\right)\right]_{ij},\quad i,j=1,\dots,n,
\end{equation*}
where $(\phi_j)_{j=1}^n$ are chosen as in Lemma \ref{exdiv} and $\delta_{ij}$ is the Kronecker delta. We have
$$
\operatorname{tr}(A)=\frac{n}{1-n}\left(\sum_{i=1}^{n}\partial_i\phi_i-p\right)=\frac{n}{1-n}(p-p)=0.
$$
We write $A=U+V$, where $U=\frac12(A^t+A),$ $V=\frac{1}{2}(A-A^t)$, and $A^t$ is the transpose of $A$. Then $U\in\mathcal{S}_0^n$, and $V$ is a skew-symmetric matrix, i.e., $V^t=-V$.
Note that if $V$ is a skew-symmetric matrix field, then $\diverg\diverg V=0$. Indeed,
$$
\diverg\diverg V=\sum_{i,j=1}\partial_i\partial_jV_{ij}=-\sum_{i,j=1}^n\partial_j\partial_i V_{ij}.
$$
Further, we have
\begin{equation*}
\begin{split}
(\diverg A)_i&=\frac{n}{1-n}\left(\sum_{j=1}^n\partial_i\partial_j\phi_j-\frac{\partial_ip}{n}\right)\\&
=\frac{n}{1-n}\partial_i\left(\sum_{j=1}^n\partial_j\phi_j-\frac{p}{n}\right)=\left(\frac{n}{1-n}\frac{n-1}{n}\right)\frac{\partial p}{\partial x_i}\\&
=-\partial_ip.
\end{split}
\end{equation*}
Thus, we get $\diverg A=-\nabla p$. Take $m=t\diverg V$. Then,
\begin{equation*}
\diverg m=t\diverg\diverg V=0
\end{equation*}
as well as
\begin{equation*}
\partial_tm+\diverg U=\diverg V+\diverg U=\diverg A=-\nabla p.
\end{equation*}
The proposition is proved.
\end{proof}
\begin{remark}
Let $p\in C_c^\infty(\R^2)$ and $\int_{\R^2}p(x_1, x_2)\dd x_1\dd x_2\neq0.$ Then there are no functions $U_1, U_2\in C_c^\infty(\R^2)$ satisfying
\begin{equation*}
\begin{cases}
\partial_1U_{1}+\partial_2U_{2}=\partial_1p,\\
\partial_1U_{2}-\partial_2U_{1}=\partial_2p.
\end{cases}
\end{equation*}
Indeed, if this were the case, then $\Delta U_2=2\partial_1\partial_2 p$, or equivalently  $-|\xi|^2\widehat{U_2}=-2\xi_1\xi_2\widehat{p}$, where $\widehat{p}(0)\neq 0$. Then $\widehat{U_2}$ and $\widehat{p}$ would be analytic functions in virtue of the compact support of $U_2$ and $p$ (see Theorem~\ref{PWS}). However, $|\xi|^2\widehat{U_2}=2\xi_1\xi_2\widehat{p}$ implies that  $|\xi|^2$ divides $\widehat{p}$, in contradiction to $\widehat{p}(0)\neq 0$.

This consideration shows that our compatibility assumption is necessary in the proof of Proposition~\ref{Prop.forU2}, as otherwise it would be impossible to construct a symmetric traceless matrix field $A$ with compact support such that $\diverg A=-\nabla p$. 
\end{remark}
\section{Geometric setup}\label{Geo}
We formulate the Euler equations as a differential inclusion and recall some well-known tools, thereby closely following~\cite{Chiodaroli,Delelissekhidieuler}. 

\begin{lemma}\label{lem3.1}
Let $m\in L^\infty(\R^n\times(0,T);\R^n)$, $U\in L^\infty(\R^n\times(0,T);\mathcal{S}^n_0)$ and \\{$q\in L^\infty(\R^n\times(0,T);\R)$} such that
\begin{equation}\label{eqs3.1}
\begin{aligned}
\diverg_xm=&0,\\
\partial_t m+\diverg_xU+\nabla_xq=&0.
\end{aligned}
\end{equation}
If $(m, U, q)$ solve \eqref{eqs3.1} and in addition there exists $\rho\in L^\infty(\R^n;\R^+)$ such
that 
\begin{equation}\label{eqs3.2}
\begin{aligned}
U&=\frac{m\otimes m}{\rho}-\frac{|m|^2}{n\rho}I_n\quad\text{a.e. in}\quad \R^n\times[0, T],\\
q&=p(\rho)+\frac{|m|^2}{n\rho}\quad\text{a.e. in}\quad \R^n\times[0, T],
\end{aligned}
\end{equation}
then $m$ and $\rho$ solve \eqref{com.euler} distributionally. Conversely, if $m$ and $\rho$ are
weak solutions of \eqref{com.euler}, then $m$, $U=\frac{m\otimes m}{\rho}-\frac{|m|^2}{n\rho}I_n$ and $q=p(\rho)+\frac{|m|^2}{n\rho}$ satisfy \eqref{eqs3.1} and \eqref{eqs3.2}.
\end{lemma}
The above Lemma is based on Lemma 3.1 in \cite{Chiodaroli}.

Next, for any given $\rho\in(0, \infty)$, we define the graph
\begin{equation*}
K_\rho:=\left\{(m, U, q)\in\R^n\times\mathcal{S}_0^n\times\R^+: U=\frac{m\otimes m}{\rho}-\frac{|m|^2}{n\rho}I_n,\hspace{0.2cm} q=p(\rho)+\frac{|m|^2}{n\rho} \right\}.
\end{equation*}
In the present setting, it is convenient to consider ``slices'' of the graph $K_\rho$ as in \cite{Chiodaroli}. For any given $\chi\in\R^+$, we thus define
\begin{equation*}
\begin{split}
K_{\rho,\chi}:=\bigg\{(m, U, q)&\in\R^n\times\mathcal{S}_0^n\times\R^+: U=\frac{m\otimes m}{\rho}-\frac{|m|^2}{n\rho}I_n,\\& q=p(\rho)+\frac{|m|^2}{n\rho},\hspace{0.2cm} |m|^2=\rho\chi \bigg\}.
\end{split}
\end{equation*}
Consider the
$(n+1)\times(n+1)$ symmetric matrix in block form
\begin{equation}\label{eq3.7}
M=\begin{pmatrix}
U+qI_n& m\\
m& 0
\end{pmatrix}.
\end{equation}
Note that, with the new coordinates $y=(x, t)\in \R^n$, the system \eqref{eqs3.1}
can be easily rewritten as $\diverg_y M = 0$. Thus, the wave cone associated with the system \eqref{eqs3.1}, i.e., the set of all states $(m,U,q)$ such that there exists $\xi\in\R^{n+1}\setminus\{0\}$ such that $(m,U,q)h(y\cdot\xi)$ satisfies~\eqref{eqs3.1} for every profile $h:\R\to\R$, is equal to
\begin{equation}\label{eq3.8}
\Lambda=\left\{(m, U, q)\in\R^n\times\mathcal{S}_0^n\times\R^+:\det\begin{pmatrix}
U+qI_n& m\\
m& 0
\end{pmatrix}=0 \right\}.
\end{equation}
For any $S\in \mathcal{S}^n$ let $\lambda_{\max}(S)$ denote the largest eigenvalue of $S$.
\begin{lemma}\label{lem3.2}
For $(\rho, m, U)\in\R^+\times\R^n\times\mathcal{S}_0^n$ let
\begin{equation}\label{eq3.9}
e(\rho, m, U):=\lambda_{\max}\left(\frac{m\otimes m}{\rho}-U \right).
\end{equation}
Then, for any given $\rho,\chi\in\R^+$,
\begin{enumerate}[(i)]
\item $e(\rho,\cdot, \cdot)\colon\R^n\times\mathcal{S}_0^n\to \R$ is convex;
\item $\frac{|m|^2}{n\rho}\leq e(\rho, m, U)$, with equality if and only if $U=\frac{m\otimes m}{\rho}-\frac{|m|^2}{n\rho}I_n$;
\item $\norm{U}_\infty\leq(n-1)e(\rho, m, U)$, where $\norm{U}_\infty$ is the operator norm of $U$;
\item the $\frac{\chi}{n}-$sublevel set of $e$ is the convex hull of $K_{\rho,\chi}$, namely,

\begin{equation}\label{eq3.10}
K_{\rho,\chi}^{co}=\left\{(m, U, q)\in\R^n\times\mathcal{S}_0^n\times\R^+: e(\rho, m, U)\leq\frac{\chi}{n}, q=p(\rho)+\frac{\chi}{n} \right\}
\end{equation}
and $K_{\rho,\chi}=K_{\rho,\chi}^{co}\cap\{|m|^2=\rho\chi\}$.
\end{enumerate}
\end{lemma}
For the proof of the Lemma we refer to~\cite[Lemma 3]{Delelissekhidieuler} and~\cite[Lemma 3.2]{Chiodaroli}.


Finally, for any $\rho,\chi\in\R^+$, we define the hyperinterior of $K_{\rho,\chi}^{co}$:
\begin{equation}\label{eq3.11}
\mbox{hint }K_{\rho,\chi}^{co}:=\left\{(m, U, q)\in\R^n\times\mathcal{S}_0^n\times\R^+: e(\rho, m, U)<\frac{\chi}{n},\hspace{0.2cm} q=p(\rho)+\frac{\chi}{n}\right\}.
\end{equation}
\section{A criterion for the existence of admissible solutions} \label{criterion}
In this section we give some criteria to recognize initial data $m^0$ for which there exist infinitely many weak admissible solutions to \eqref{comm.euler}. Again, this section closely follows~\cite{Delelissekhidieuler} and~\cite{Chiodaroli}.
\begin{proposition}\label{Prop4.1}
Let $\Omega\subset \R^n$ be a bounded open set, $\rho_0\in C^1(\R^n)$ be a given density function with $\rho_0(x)=\overline{\rho}=\text{const}$ for $x\in \R^n\backslash\Omega$ and let $T>0$ be any finite time and  $\Omega'\supset\Omega$ be a bounded open set. Assume that there exist $(m_0, U_0, q_0)$ continuous solutions of
\begin{equation*}
\begin{cases}
\diverg_x m_0=0, \\
\partial_t m_0+\diverg_x U_0+\nabla_x q_0=0\quad\text{on}\quad\R^n\times(0, T)\\
\end{cases}
\end{equation*}
with $m_0\in C([0, T];H_w(\R^n))$, $\supp\left(m_0(\cdot, t), U_0(\cdot, t) \right)\subset\subset\Omega'$ for all $t\in (0, T),$ and a function $\chi\in C^\infty([0, T]; \R^+)$ such that
$$
e\left(\rho_0(x), m_0(x, t), U_0(x, t)\right)<\frac{\chi(t)}{n}
$$
for all $(x,t)\in \R^n\times(0, T)$,
$$
q_0(x, t)=p(\rho_0(x))+\frac{\chi(t)}{n}
$$
for all $(x, t)\in\R^n\times(0, T)$.
Then there exist infinitely many weak solutions $(\rho, m)$ of the system \eqref{com.euler} in $\R^n\times[0, T)$ with density $\rho(x)=\rho_0(x)$ and such that
\begin{equation}\label{4.4-4.6}
\begin{split}
&m\in C([0, T]; H_w(\R^n)),\\
&m(\cdot, t)=m_0(\cdot, t)\quad\text{for}\quad t=0, T\quad\text{and for a.e. $x\in\R^n$,}\\
&|m(x,t)|^2=\rho_0(x)\chi(t)\mathbbm{1}_{\Omega'}
\quad\text{for a.e}\quad (x,t)\in \R^n\times(0, T).
\end{split}
\end{equation}
\end{proposition}
\subsection{The space of subsolutions}
Let $m_0$ be a vector field as in Proposition \ref{Prop4.1} with associated modified pressure $q_0$, 
$$q_0=p(\rho_0)+\frac{\chi(t)}{n},$$
where $\rho_0$ and $\chi$ are given functions as in the assumptions of Proposition \ref{Prop4.1}. Consider momentum fields $m\colon\R^n\times[0, T]\to\R^n$ which satisfy
\begin{equation}\label{divfree4.7}
\diverg m=0,
\end{equation}
the initial conditions
\begin{equation}\label{con4.8}
\begin{split}
&m(x, 0)=m_0(x, 0),\\
&m(x, T)=m_0(x, T),\\
&\supp m(\cdot, t)\subset\Omega'\quad\text{for all}\quad t\in (0, T)
\end{split}
\end{equation}
 and such that there exists a continuous matrix field $U\colon\R^n\times (0, T)\to\mathcal{S}_0^n$ with
\begin{equation}\label{con4.9}
\begin{split}
&e\left(\rho_0(x), m(x, t), U(x, t)\right)<\frac{\chi(t)}{n}\quad\text{for all}\quad(x, t)\in \Omega'\times (0, T),\\
&\supp\left(U(\cdot, t)\right)\subset\Omega'\quad\text{for all}\quad t\in (0, T),\\
&\partial_t m+\diverg U+\nabla q_0=0\quad\text{in}\quad\R^n\times[0, T].
\end{split}
\end{equation}

\begin{definition}\label{subsol}
Let $X_0$ be the set of such momentum fields:
\begin{equation*}
X_0=\set{m\in C^0\left((0, T); C_c(\R^n)\right)\cap C\left([0, T]; H_w(\R^n)\right):\eqref{divfree4.7},\eqref{con4.8}, \eqref{con4.9}\quad\text{are satisfied}}
\end{equation*} and let $X$ be the closure of $X_0$ in $C\left([0, T]; H_w(\R^n)\right)$. Then $X_0$ is called the space of strict subsolutions.
\end{definition}
Let
\begin{equation*}
G=\sup_{t\in[0, T]}\chi(t)\int_{\Omega'}\rho_0(x)\dd x.
\end{equation*}
Since for any $m\in X_0$ with associated matrix field $U$, we have that (see Lemma \ref{lem3.2} (ii))
\begin{equation*}
\begin{split}
\int_{\R^n}|m(x, t)|^2\dd x&=\int_{\Omega'}|m(x, t)|^2\dd x\le\int_{\Omega'} n\rho_0(x)e\left(\rho_0(x), m(x, t), U(x, t)\right)\dd x\\&
\le\chi(t)\int_{\Omega'}\rho_0(x)\dd x\leq
G\quad\text{for all }t\in [0, T].
\end{split}
\end{equation*}
We can observe that $X_0$ consists of functions $m\colon[0, T]\to H(\R^n)$ taking values in a bounded subset $B=B_G(0)$ of $H(\R^n)$. Without loss of generality, we can assume that $B$ is weakly closed. Then $B$ is metrizable in its weak topology and, if we let $d_B$ be a metric on $B$ inducing the weak topology, we have that $(B, d_B)$ is a compact metric space. Moreover, we can define on $Y:=C([0, T], (B, d_B))$ a metric $d$ naturally induced by $d_B$ via
\begin{equation}\label{eq4.11}
d(f_1, f_2)=\max_{t\in [0, T]}d_B\left(f_1(\cdot, t), f_2(\cdot, t)\right).
\end{equation}
Note that the topology induced on $Y$ by $d$ is equivalent to the topology of $Y$ as a subset of $C([0, T]; H_w)$. In addition, the space $(Y, d)$ is complete. Finally, $X$ is the closure in $(Y, d)$ of $X_0$ and hence $(X, d)$ is as well a complete metric space. If $m\in X$ then $\supp(m(\cdot, t))\subset \overline{\Omega'}$ for any $t\in[0, T]$.

\begin{lemma}\label{lem4.3}
If $m\in X$ is such that $|m(x, t)|^2=\rho_0(x)\chi(t)\mathbbm{1}_{\Omega'}$ for a.e. $(x, t)\in\R^n\times(0, T)$, then the pair $(\rho_0, m)$ is a weak solution of \eqref{com.euler} in $\R^n\times[0, T)$ satisfying \eqref{4.4-4.6}.
\end{lemma}
\begin{proof}
Let $m\in X$ be such that
\begin{equation*}
|m(x, t)|^2=\rho_0(x)\chi(t)\quad\text{for a.e.}\quad(x, t)\in\Omega'\times(0, T).
\end{equation*}
By density of $X_0$, there exists a sequence $\set{m_k}\subset X_0$ such that $m_k\overset{d}{\longrightarrow} m$ in $X$. For any $m_k\in X_0$, let $U_k$ be the associated smooth matrix field enjoying the properties~\eqref{con4.9}. By using Lemma \ref{lem3.2} (iii) and
\begin{equation*}
e\left(\rho_0(x), m_k(x, t), U_k(x, t)\right)<\frac{\chi(t)}{n},
\end{equation*}
the following pointwise estimate holds for the sequence $\set{U_k}$:
\begin{equation*}
|U_k(x, t)|\le(n-1)e\left(\rho_0(x), m_k(x, t), U_k(x, t)\right)<\frac{(n-1)\chi(t)}{n}.
\end{equation*}
Consequently,
\begin{equation*}
\norm{U_k}_\infty\le(n-1)\norm{e\left(\rho_0(\cdot), m_k(\cdot, t), U_k(\cdot, t)\right)}_\infty<\frac{(n-1)\chi(t)}{n}.
\end{equation*}
As a consequence, $\set{U_k}$ is uniformly bounded in $L^\infty(\R^n\times(0,T))$, and by possibly extracting a subsequence, we have
\begin{equation*}
U_k\overset{\ast}{\rightharpoonup}U\quad\text{in}\quad L^\infty(\R^n\times(0,T)).
\end{equation*}
Following \cite{Chiodaroli}, $\overline{\mbox{hint }K_{\rho_0, \chi}^{co}}=K_{\rho_0, \chi}^{co}$ is a convex and compact set by Lemma \ref{lem3.2} (i), (ii), (iii). Hence $m\in X$ with associated matrix field $U$ solves
\begin{equation*}
\begin{split}
&\diverg_xm=0,\\
&\partial_t m+\diverg_x U+\nabla_xq_0=0\quad\text{on}\quad\R^n\times[0, T]
\end{split}
\end{equation*}
and $(m, U, q_0)$ takes values in $K_{\rho_0,\chi}^{co}$ almost everywhere. If, in addition, $|m(x,t)|^2=\rho_0\chi(t)\mathbbm{1}_{\Omega'}$, then $(m, U, q_0)(x, t)\in K_{\rho, \chi}$ a.e. in $\R^n\times[0, T]$ (because $K_{\rho, \chi}^{co}\cap\set{|m|^2=\rho \chi}=K_{\rho, \chi}$). Lemma \ref{lem3.1} allows us to conclude that $(\rho_0, m)$ is a weak solution of \eqref{com.euler}.
Finally, since $m_k\to m$ in $C([0, T]; H_w(\R^n))$ and $|m(x, t)|^2=\rho_0(x)\chi(t)\mathbbm{1}_{\Omega'}$ for almost every $(x, t)\in \R^n\times(0, T)$, we see that $m$ satisfies also \eqref{4.4-4.6}. Lemma \ref{lem4.3} is proved.
\end{proof}

\begin{lemma}\label{lem4.4}
The identity map $I\colon(X, d)\to L^2([0, T]; H
)$ defined by $m\mapsto m$ is a Baire-1 map, and therefore the set of points of continuity is residual in $(X, d)$.
\end{lemma}
The proof of this lemma relies on Baire category arguments \cite{Oxtoby} which can be easily adapted to our case from  Lemma 4.4 in \cite{Chiodaroli}.
\subsection{Proof of Proposition \ref{Prop4.1}}
We aim to show that all points of continuity of the identity map correspond to solutions of \eqref{com.euler} satisfying the requirements of Proposition \ref{Prop4.1}. Continuity of the identity map (see Lemma \ref{lem4.4}) will then allow us to prove Proposition \ref{Prop4.1}, once we know that the cardinality of $X$ is infinite. In light of Lemma \ref{lem4.3}, for our purpose it suffices to prove the following claim.\\
{\bf Claim:} If $m\in X$ is a point of continuity of $I$, then
\begin{equation}\label{eq4.14}
|m(x, t)|^2=\rho_0(x)\chi(t)\mathbbm{1}_{\Omega'}\text{ for a.e } (x,t)\in \R^n\times(0, T).
\end{equation}
As in~\cite{Chiodaroli},~\eqref{eq4.14} is equivalent to 
\begin{equation}
\norm{m}_{L^2(\Omega'\times[0, T])}=\left(\int_{\Omega'}\int_0^T\rho_0(x)\chi(t)\dd t\dd x\right)^{\frac12},
\end{equation}
since for any $m\in X$, we have
\begin{equation*}
|m(x, t)|^2\le \rho_0(x)\chi(t)\mathbbm{1}_{\Omega'}(x)
\end{equation*}
for almost all $(x, t)\in \Omega'\times(0, T)$.
Thanks to this remark, the claim is reduced to the following lemma.
\begin{lemma}\label{lem4.5}
Let $\rho_0, \chi$ be given functions as in Proposition \ref{Prop4.1}. Then there exists a constant $\beta=\beta(n)>0$ such that, given $m\in X_0$, there exists a sequence $\set{m_k}\subset X_0$ with the following properties:
\begin{equation}\label{eq4.15}
\begin{split}
\norm{m_k}_{L^2(\Omega'\times[0, T])}&\geq\norm{m}_{L^2(\Omega'\times[0, T])}+\beta\left(\int_{\Omega'}\int_{0}^{T}\rho_0(x)\chi(t)d t\dd x-\norm{m}^2_{L^2(\Omega'\times[0, T])} \right)^2
\end{split}
\end{equation}
and $m_k\to m$ in $C([0, T], H_w(\Omega'))$.
\end{lemma}
The proof of the above Lemma can be observed after minor changes in Lemma 4.5 from~\cite{Chiodaroli}. Now, let us show how Lemma \ref{lem4.5} implies the claim. Let $m\in X$ be a point of continuity of $I$. Owing to the density of $X_0$ in $X$ and Lemma~\ref{lem4.5}, there exist sequences $\{m_k\}, \{\widetilde{m}_k\}\subset X_0$ such that  $m_k\overset{d}{\longrightarrow}m$, $\widetilde{m}_k\overset{d}{\longrightarrow}m$, and
\begin{equation*}
\begin{split}
\liminf_{k\to\infty}&\norm{\widetilde{m}_k}_{L^2(\Omega'\times[0, T])}^2\geq\liminf_{k\to\infty}\Bigg(\norm{m_k}^2_{L^2(\Omega'\times[0, T])}\\&+\beta\left(\int_{\Omega'}\int_{0}^{T}\rho_0(x)\chi(t)\dd t\dd x-\norm{m_k}^2_{L^2(\Omega'\times[0, T])} \right)^2\Bigg).
\end{split}
\end{equation*}
By the assumption, $I$ is continuous at $m$, which implies that both $m_k$ and $\widetilde{m}_k$ converge strongly to $m$ and
\begin{equation*}
\begin{split}
\norm{m}_{L^2(\Omega'\times[0, T])}&\geq\norm{m}_{L^2(\Omega'\times[0, T])}\\&\quad+\beta\left(\int_{\Omega'}\int_{0}^{T}\rho_0(x)\chi(t)\dd t\dd x-\norm{m}^2_{L^2(\Omega'\times[0, T])} \right)^2.
\end{split}
\end{equation*}
Therefore $\norm{m}^2_{L^2(\Omega'\times[0, T])}=\int_{\Omega'}\int_{0}^{T}\rho_0(x)\chi(t)\dd t\dd x$ and the claim is proved.\qed

\section{Construction of suitable initial data}
The aim of this section is to prove the existence of a subsolution in the sense of Definition \ref{subsol} for which we apply Proposition \ref{Prop4.1} to generate infinitely many solutions.
\begin{proposition}\label{p7.1}
Let $\rho_0\in C^1(\R^n; \R^+)$ be a function satisfying the conditions $\rho_0>0$ on $\R^n$ and $\rho_0(x)=\overline{\rho}$ constant on $\R^n\backslash\Omega$. Let $p(\rho_0)$ be $C^1$ function such that
$$
\int_{\Omega}p(\rho_0)\dd x=p(\overline{\rho})|\Omega|,
$$
$T>0$, and $\Omega'$ a bounded locally Lipschitz domain with $\Omega'\supset\supset\Omega$.

Then there exist $\widetilde{U}\colon\R^n\to\mathcal{S}^n_0$ and $\widetilde{m}(\cdot, t)\colon\R^n\to\R^n$ such that
\begin{equation*}
\partial_t\widetilde{m}+\diverg_x\widetilde{U}+\nabla_x{q}_0=0\quad\text{on}\quad\R^n\times\R,
\end{equation*}
\begin{equation*}
\supp(\widetilde{m}(\cdot, t),\widetilde{U}(\cdot))\subset\Omega'\quad\text{for any}\quad t\in[0, T],
\end{equation*}
\begin{equation}\label{7.2}
e(\rho_0(x), \widetilde{m}(x, t),\widetilde{U}(x, t))<\frac{{\chi}(t)}{n}\quad\text{for all }(x,t)\in\R^n\times[0, T),
\end{equation}
for any continuous function ${\chi}:[0,T)\to\R$ such that 
\begin{equation*}
{\chi}(t)>n{\lambda}(t):=n\norm{e(\rho_0(\cdot)), \widetilde{m}(\cdot, t), \widetilde{U}(\cdot)}_{L^\infty(\Omega')}
\end{equation*}
for every $t\in[0,T)$ and for
\begin{equation*}
{q}_0(x, t):=p(\rho_0(x))+\frac{{\chi}(t)}{n}\quad\text{for all }x\in\R^n\times\R.
\end{equation*}

\end{proposition}
\begin{proof}
Let $\eps>0$ be so small that $\overline{\Omega^\eps}\subset\Omega'$ and 
$p^\eps$ be given as in Lemma \ref{LemPWS}, i.e., $p^\eps=p(\rho_0)-p(\rho_0)*\omega^\eps$.
Then, this Lemma implies that there exists $u\in C^{2,\alpha}_c(\Omega')$ such that $\Delta u=p^\eps$ and $\supp u\subset\overline{\Omega^\eps}$.
Now, 
we define a matrix field by
\begin{equation*}
U_{ij}^{(1)}=-\frac{n}{n-1}\frac{\partial^2u}{\partial x_i\partial x_j}\quad i\neq j,
\end{equation*}
\begin{equation*}
U_{ii}^{(1)}=-\frac{n}{n-1}\frac{\partial^2u}{\partial x_i^2}+\frac{p^\eps}{n-1}.
\end{equation*}
Then  $U^{(1)}$ is obviously symmetric and
\begin{equation*}
\operatorname{tr}(U^{(1)})=-\frac{n}{n-1}(\Delta u-p^\eps)=0.
\end{equation*}
We show that
\begin{equation*}
\diverg U^{(1)}=-\nabla p^\eps.
\end{equation*}
Indeed, for fixed $i$, we have
\begin{equation*}
\begin{split}
-\diverg_x U_{i\cdot}^{(1)}&=\frac{n}{n-1}\frac{\partial^3 u}{\partial x_i^3}-\frac{\frac{\partial}{\partial x_i}p^\eps}{n-1}+\frac{n}{n-1}\frac{\partial}{\partial x_i}\Delta u-\frac{n}{n-1}\frac{\partial^3 u}{\partial x_i^3}\\&=\frac{\partial}{\partial x_i}p^\eps.
\end{split}
\end{equation*}
Next, note that $p_1*\omega^\eps$ is a smooth function with compact support satisfying the condition of Proposition \ref{Prop.forU2}, i.e.,
\begin{equation*}
\int_{\R^n}p_1*\omega^\eps(x)\dd x=0.
\end{equation*}
So Proposition \ref{Prop.forU2} implies that we can find a matrix field $U^{(2)}\in \mathcal{S}_0^n$ and a vector field $\widetilde{m}$ satisfying the following conditions:
\begin{itemize}
\item $(\widetilde{m}(t), U^{(2)})\in C_c^\infty(\Omega')$ for every $t\in\R$,
\item $\diverg_x \widetilde{m}=0$,
\item $\partial_t \widetilde{m}+\diverg_x U^{(2)}+\nabla(p_1*\omega^\eps)=0$.
\end{itemize}
Next, we define $\widetilde{U}:=U^{(1)}+U^{(2)}$, then we have
\begin{equation*}
\supp(\widetilde{m}(\cdot, t),\widetilde{U}(\cdot))\subset\Omega'\quad \text{for every}\quad t\in\R.
\end{equation*}
Besides, recalling
\begin{equation*}
{\lambda}(t)=\norm{e(\rho_0(\cdot)), \widetilde{m}(\cdot, t), \widetilde{U}(\cdot)}_{L^\infty(\Omega')}=\norm{\lambda_{\max}\left(\frac{\widetilde{m}\otimes \widetilde{m}}{\rho_0}-\widetilde{U} \right)}_{L^\infty(\Omega')},
\end{equation*}
then since $\widetilde{m}$ is linearly $t$-dependent and $\widetilde{U}$ is independent of $t$, we have 
\begin{equation*}
 |{\lambda}(t)|\leq C_1 t^2+C_2\quad \text{for some} \quad C_1, C_2>0 \quad \text{for any}\quad t\in\R.
\end{equation*}
This indicates that we can choose any continuous function ${\chi}$ on $\R$ satisfying ${\chi}(t)>n{\lambda}(t)$ to ensure \eqref{7.2}. 
So we get $(\widetilde{m},\widetilde{U})$ satisfying the required conditions, thereby completing the proof of Proposition \ref{p7.1}.
\end{proof}

\begin{proposition}\label{p7.3}
Let $\rho_0, p$ be continuously differentiable functions as in Proposition \ref{p7.1} and $\Omega'\supset\supset\Omega$ bounded Lipschitz. 
Also, let $T>0$ be any given time and $(\widetilde m,\widetilde{U},{q}_0)$ and ${\chi}$ be as in Proposition~\ref{p7.1}. Then there exists a pair $({m}_0, {U}_0)$ solving the system
\begin{equation}\label{p7.3.*}
\begin{split}
\diverg_x{m}_0&=0\\
\partial_t{m}_0+\diverg_x{U}_0+{q}_0&=0
\end{split}
\end{equation}
distributionally on $\R^n\times(0,T)$ enjoying the following properties:
$({m}_0, {U}_0, {q}_0)$ is continuous in $\R^n\times(0,T]$ and ${m}_0\in C([0,T]; H_w(\R^n))$,
\begin{equation}\label{p7.3.2}
\supp({m}_0(\cdot,0), {U}_0(\cdot,0))\subset\overline{\Omega'},
\end{equation}
\begin{equation}\label{p7.3.3}
\supp({m}_0(\cdot, t), {U}_0(\cdot,t))\subset\subset\Omega'\quad\text{for all }t\neq 0,
\end{equation}
\begin{equation}\label{p7.3.4}
{q}_0(x,t)=p(\rho_0(x))+\frac{{\chi}(t)}{n}\quad\text{for all }(x,t)\in\R^n\times[0,T],
\end{equation}
\begin{equation}\label{p7.3.5}
e(\rho_0(x), {m}_0(x,t), {U}_0(x,t))<\frac{{\chi}(t)}{n}\quad\text{for all }(x,t)\in \R^n\times(0, T].
\end{equation}
Furthermore,
\begin{equation}\label{p7.3.6}
|{m}_0(x,0)|^2=\rho_0(x){\chi}(0)\quad\text{a.e in}\quad\Omega'.
\end{equation}
\end{proposition}
\begin{proof}
Identity~\eqref{p7.3.4} is already satisfied by definition of $q_0$.

In analogy with Definition \ref{subsol} (see also \cite{Delelissekhidieuler} and \cite{Chiodaroli}), we consider the space $X_0$ defined as the set of continuous vector fields $m\colon\R^n\times[0,T)\to\R^n$ to which there exists a continuous matrix field $U:\R^n\times[0,T)\to\mathcal{S}_0^n$ such that
\begin{equation}\label{p7.3.7}
\begin{split}
\diverg_x{m}&=0\\
\partial_tm+\diverg_x U+{q}_0&=0
\end{split}
\end{equation}
\begin{equation*}
\supp(m-\widetilde{m})\subset\Omega'\times\left[0, \frac{T}{2}\right)
\end{equation*}
\begin{equation*}
U(\cdot, t)=\widetilde{U}(\cdot)\quad\text{for all  }t\in\left[\frac T2, T\right)
\end{equation*}
and
\begin{equation}\label{eq7.15}
e(\rho_0(x), m(x, t), U(x, t))<\frac{{\chi}(t)}{n}\text{ for all }(x, t)\in \Omega'\times[0, T).
\end{equation}
Note that $\widetilde{m}\in X_0$, where $\widetilde{m}$ is the vector field given by Proposition \ref{p7.1}. 
As before, we set $d$ to be a metrization of the convergence in $C([0,T);L^2_w(\Omega'))$, and $X$ to be the closure of $X_0$ w.r.t.\ this topology.

Now following \cite{Chiodaroli,Delelissekhidieuler} we use the following claim which can be verified by minor modifications in the proof of Lemma \ref{lem4.5}:

{\bf Claim:} Let $\emptyset\neq\Omega_0\subset\subset\Omega'$ be a given domain and $\delta>0$. For every $\alpha>0$ there exists $\beta>0$ such that the following holds: Let $m\in X_0$ with associated matrix field $U$ be such that
$$
\int_{\Omega_0}\left[|m(x, 0)|^2-(\rho_0(x){\chi}(0))\right]\dd x<-\alpha.
$$
Then, there exists a sequence $m_k\in X_0$ with associated matrix field $U_k$ such that
\begin{equation*}
\supp(m_k-m, U_k-U)\subset \Omega_0\times[0, \delta],
\end{equation*}
\begin{equation*}
m_k\overset{d}{\longrightarrow}m,
\end{equation*}
and
\begin{equation*}
\liminf_{k\to\infty}\int_{\Omega_0}|m_k(x,0)|^2\dd x\geq\int_{\Omega_0}|m(x, 0)|^2\dd x+\beta\alpha^2.
\end{equation*}

Next, fix an exhausting sequence of bounded open subsets $\Omega_k\subset\Omega_{k+1}\subset\Omega'$, each compactly contained in $\Omega'$, and such that $|\Omega_{k+1}\backslash \Omega_k|\leq 2^{-k}$. Let also $\eta_\eps$ be a standard mollifying kernel in $\R^n$ with $\supp\eta_\eps\subset B_\eps(0)$. In view of the claim above, we construct inductively a
sequence of momentum fields $m_k\in X_0$, associated matrix fields $U_k$ and a sequence of numbers $\gamma_k < 2^{-k}$ as follows.

Firstly let $m_1(x,t)=\widetilde{m}(x,t)$, $U_1(x,t)=\widetilde{U}(x)$ for all $(x,t)\in\R^n\times[0,T)$. After obtaining $(m_1, U_1),\dots,(m_k, U_k)$ and
$\gamma_k,\dots,\gamma_{k-1}$, we choose $\gamma_k<2^{-k}$ in such a way that
\begin{equation}\label{p7.3.11}
\sup_{t\in[0,T)}\|m_k-m_k*\eta_{\gamma_k}\|_{L^2(\Omega')}<2^{-k}.
\end{equation}
Next, we set
\begin{equation*}
\alpha_k=-\int_{\Omega_k}\left[|m_k(x,0)|^2-\rho_0(x){\chi}(0)\right]\dd x.
\end{equation*}
Note that because of \eqref{eq7.15} we have $\alpha_k>0$. Then we apply the claim with $\Omega_k$, $\alpha=\alpha_k$ and $\delta=2^{-k}T$ to obtain $m_{k+1}\in X_0$ and an associated smooth matrix field $U_{k+1}$ such that
\begin{equation}\label{p7.3.12}
\supp(m_{k+1}-m_k, U_{k+1}-U_k)\subset\Omega_k\times[0, 2^{-k}T],
\end{equation}
\begin{equation}\label{p7.3.13}
d(m_{k+1}, m_k)<2^{-k},
\end{equation}
\begin{equation}\label{p7.3.14}
\int_{\Omega_k}|m_{k+1}(x,0)|^2\dd x\geq\int_{\Omega_{k}}|m_{k}(x,0)|^2\dd x+\beta\alpha_k^2.
\end{equation}
Since $d$ induces the topology of $C([0,T); L_w^2(\Omega'))$ we can additionally prescribe that
\begin{equation}\label{p7.3.15}
\|(m_k-m_{k+1})*\eta_{\gamma_j}\|_{L^2(\Omega')}<2^{-k}\quad\text{for all }j\leq k\text{ for }t=0,
\end{equation}  because
\begin{equation*}
\|(m_k-m_{k+1})*\eta_{\gamma_j}(t=0)\|_{L^2(\Omega')}\leq d(m_{k+1}, m_k)<2^{-k}.
\end{equation*}
In view of \eqref{p7.3.13}, we derive the existence of a function ${m}_0\in C([0,T); H_w(\Omega'))$ such that $$m_k\overset{d}{\longrightarrow}{m}_0.$$
From \eqref{p7.3.12} we see that for any compact subset $A$ of $\Omega'\times(0, T)$ there exists $k_0$ such that $(m_k, U_k)|_A=(m_{k_0}, U_{k_0})|_A$ for all $k>k_0$.
So $(m_k, U_k)$ converges in $C_{loc}(\Omega'\times(0, T))$ to a  continuous pair $({m}_0, {U}_0)$ solving equation \eqref{p7.3.*} in $\R^n\times(0, T)$ and  satisfying~\eqref{p7.3.2},\eqref{p7.3.3},\eqref{p7.3.4},\eqref{p7.3.5}. In order to show that \eqref{p7.3.6} also holds for ${m}_0$, we observe that \eqref{p7.3.14} yields
$$
\alpha_{k+1}\leq \alpha_k-\beta\alpha_k^2+C|\Omega_{k+1}\backslash\Omega_k|\leq\alpha_k-\beta\alpha_k^2+C2^{-k},
$$ which implies that $\alpha_k\to 0$ as $k\to \infty$.
Furthermore,
\begin{equation*}
\begin{split}
0&\geq \int_{\Omega'}\left[|m_k(x,0)|^2-\rho_0(x){\chi}(0)\right]\dd x\\&
\geq-(\alpha_k+C_1|\Omega'\backslash \Omega_k|)\\&
\geq-(\alpha_k+C_12^{-(k-1)}),
\end{split}
\end{equation*}
because $|\Omega'\backslash\Omega_k|=\cup_{j=k}^\infty|\Omega_{j+1}\backslash\Omega_j|=\sum_{j=k}^\infty 2^{-j}=2^{-(k-1)}$. The latter two observations imply that
\begin{equation}\label{p7.3.17}
\lim_{k\to\infty}\int_{\Omega'}\left[|m_k(x, 0)|^2-\rho_0(x){\chi}(0)\right]\dd x=0.
\end{equation}
On the other hand, owing to \eqref{p7.3.11} and \eqref{p7.3.15} we can write for $t=0$ and for every $k$
\begin{equation*}
\norm{m_k-{m}_0}_{L^2}\leq\norm{m_k-m_k*\eta_{\gamma_k}}_{L^2}+\norm{m_k*\eta_{\gamma_k}-\overline{m}*\eta_{\gamma_k}}_{L^2}+\norm{\overline{m}*\eta_{\gamma_k}-\overline{m}}_{L^2}.
\end{equation*}
By construction, $\norm{m_k-m_k*\eta_{\gamma_k}}_{L^2}<2^{-k}$ and
\begin{equation*}
\begin{split}
\norm{m_k*\eta_{\gamma_k}-{m}_0*\eta_{\gamma_k}}_{L^2}&\leq\sum_{j=0}^\infty\norm{m_{k+j}*\eta_{\gamma_k}-m_{k+j+1}*\eta_{\gamma_k}}_{L^2}\\&\leq\sum_{j=0}^\infty2^{-(k+j)}=2^{-(k-1)}.
\end{split}
\end{equation*}
Thus we have
$$
\norm{m_k-{m}_0}_{L^2}\leq 2^{-k}+2^{-(k-1)}+\norm{{m}_0*\eta_{\gamma_k}-{m}_0}_{L^2}.
$$
Hence $\norm{m_k-{m}_0}_{L^2}\to 0$ as $k\to\infty$, i.e., $m_k(\cdot,0)\to{m}_0(\cdot, 0)$ strongly in $H(\Omega')$ as $k\to\infty$. Combining this with \eqref{p7.3.17} implies
$$
|{m}_0(x, 0)|^2=\rho_0(x){\chi}(0)\quad\text{for a.e. }x\in\Omega'
$$ which completes the proof of Proposition \ref{p7.3}.
\end{proof}
\section{Proof of the main results}\label{main}
\subsection{Proof of Theorem \ref{Th2}} Let $T$ be any finite positive time and $\rho_0\in C^1(\R^n)$ be a given density function as in Theorem \ref{Th2}. Further assume that $({m}_0, {U}_0, {q}_0)$ and $\chi$ is given by Proposition \ref{p7.3}. 
Then, $(m_0, U_0, q_0, \chi)$ fulfills the assumptions of Proposition \ref{Prop4.1}. Thus, there exist infinitely many solutions $m\in C([0, T), H_w(\R^n))$ of \eqref{com.euler} in $\R^n\times[0, T)$ with density $\rho_0$ such that
$$
m(x, 0)={m}_0(x, 0)\quad\text{for a.e.  }x\in \Omega'
$$
and
$$
|m(x, t)|^2=\rho_0(x)\chi(t)\mathbbm{1}_{\Omega'}\quad\text{a.e. in  }\R^n\times(0, T).
$$
Since  $|m_0(x, 0)|^2=\rho_0(x)\chi(0)$ a.e. in $\Omega'$ as well, it is enough to define $m^0(x)=m_0(x, 0)$ to satisfy \eqref{eq2.10} and hence conclude the proof. \qed
\subsection{Proof of Theorem \ref{Th3}} Under the assumptions of Theorem \ref{Th2}, we have shown the existence of a bounded initial momentum $m^0$ allowing for infinitely many solutions $m\in C([0, T]; H_w(\R^n))$ of \eqref{com.euler} on $\R^n\times[0, T)$ with density $\rho_0$. Moreover, according to Proposition \ref{p7.1}, for an arbitrary continuous function $\chi\colon\R\to\R^+$ with $\chi>n{\lambda}>0$, we have the following equalities:
\begin{equation}\label{eq16}
|m(x,t)|^2=\rho_0(x)\chi(t)\mathbbm{1}_{\Omega'}\quad\text{a.e. in  }\R^n\times[0, T),
\end{equation}
and in particular
\begin{equation}\label{eq17}
|m^0(x)|^2=\rho_0(x)\chi(0)\quad\text{a.e. in  }\Omega'.
\end{equation}
Now, we claim that there exist constants $C_1, C_2>0$ such that choosing the  function $\chi(t)>n{\lambda}$ on $[0, T)$ among solutions of the differential inequality
\begin{equation}\label{eq**}
\chi'(t)\leq-C_1\chi^{\frac12}(t)-C_2\chi^{\frac32}(t)
\end{equation}
yields weak solutions $(\rho_0,m)$ of \eqref{com.euler} (obtained through Theorem \ref{Th2}) that will also satisfy the admissibility condition \eqref{eq2.5ch} on $\R^n\times[0, T)$.

Suppose for the moment this claim is true. Then, one may simply choose $\chi$ to be the solution of the ordinary differential equation
\begin{equation*}
\chi'(t)=-C_1\chi^{\frac12}(t)-C_2\chi^{\frac32}(t)
\end{equation*}
with initial condition $\chi(0)=\chi^0$ sufficiently large so that $\chi$ which will remain greater than $n{\lambda}$ up to some positive time $\overline T$.

%
%

Finally, we aim to prove the claim. Since $m\in C([0, T]; H_w(\R^n))$ is divergence  free and fulfills \eqref{eq16},\eqref{eq17} and  $\rho_0$ is time independent, \eqref{eq2.5ch} reduces to the following inequality
\begin{equation}\label{eq8.29}
\frac{1}{2}\chi'(t)+m\cdot\nabla\left(\eps(\rho_0)+\frac{p(\rho_0)}{\rho_0}\right)+\frac{\chi(t)}{2}m\cdot\nabla\left(\frac{1}{\rho_0}\right)\leq 0
\end{equation}
intended in the sense of distributions on $\R^n\times[0, T)$. As $\rho_0\in C^1(\R^n)$ is bounded, there exists a constant $C_0^2$ with $\rho_0\leq C_0^2$ on $\R^n$, whence (see \eqref{eq16},\eqref{eq17})
\begin{equation}\label{mbound}
|m(x,t)|\leq C_0\sqrt{\chi(t)}\quad\text{a.e.\ on $\Omega'\times[0, T)$}. 
\end{equation}
Analogously we can find constants $c_1,c_2>0$ with
\begin{equation}\label{eq18}
\left|\nabla\left(\eps(\rho_0)+\frac{p(\rho_0)}{\rho_0}\right) \right|\leq c_1\quad\text{a.e.\ in $\Omega$}
\end{equation}
\begin{equation}\label{eq19}
\left|\nabla\left(\frac{1}{\rho_0}\right)\right|\leq c_2\quad\text{a.e.\ in $\Omega$}.
\end{equation}
As a consequence of \eqref{mbound},\eqref{eq18} and \eqref{eq19},  \eqref{eq8.29} holds as soon as $\chi$ satisfies
\begin{equation*}
\chi'(t)\leq-2c_1C_0\chi^{\frac12}(t)-c_2C_0\chi^{\frac32}(t)\quad\text{on }[0, T).
\end{equation*}
Therefore, by choosing $C_1:=2c_1C_0$ and $C_2:=c_2C_0$ we can conclude the proof of the claim. 
\qed

\subsection{Proof of Corollary \ref{Th1}} In analogy with \cite{Chiodaroli} note that the proof of Corollary~\ref{Th1} relies on Theorems \ref{Th2}-\ref{Th3}. Given a continuously differentiable initial density
$\rho^0$ we apply Theorems \ref{Th2}-\ref{Th3} for $\rho_0(x):= \rho^0(x)$ thus obtaining a positive time $\overline{T}$ (depending on $\norm{\rho^0}_{C^1}$) and a bounded initial momentum $m^0$ for which there exist infinitely many
solutions $m\in C([0, T];H_w(\R^n))$ of \eqref{com.euler} on $\R^n\times[0, \overline{T})$ with density $\rho^0$ and, additionally, the following holds:
\begin{equation}\label{eq20}
|m(x,t)|^2=\rho_0(x)\chi(t)\mathbbm{1}_{\Omega'}\quad\text{a.e. in  }\R^n\times[0, \overline{T}),
\end{equation}
\begin{equation}\label{eq21}
|m^0(x)|^2=\rho_0(x)\chi(0)\mathbbm{1}_{\Omega'}\quad\text{a.e. in  }\R^n
\end{equation}
for a suitable smooth function $\chi:[0, \overline{T}]\to\R^+$. Now, define $\rho(x, t) =
\rho_0(x)\mathbbm{1}_{[0, \overline{T})}(t)$. This indicates that \eqref{eq2.4ch} holds. Similarly to \cite{Chiodaroli} we observe that
$\rho$ is independent of $t$ and $m$ is weakly divergence-free for almost every
$0 < t < \overline{T}$. Therefore, the pair $(\rho, m)$ is a compactly supported weak solution of \eqref{comm.euler} with
initial data $(\rho^0, m^0)$. In the end, note that each solution obtained is also admissible. In fact, for $\rho(x, t) = \rho_0(x)\mathbbm{1}_{[0,\overline{T})}(t)$, \eqref{eq2.5ch} is
assured by Theorem \ref{Th3}. Corollary  \ref{Th1} is proved.\qed

\end{document}